\documentclass[letterpaper,reqno]{amsart}

\usepackage{amsmath} 
\usepackage{amsfonts} 
\usepackage{amssymb} 
\usepackage{epsfig} 
\usepackage{graphicx} 
\usepackage{multirow} 
\usepackage{verbatim} 
\usepackage{rotating} 
\usepackage[all]{xy} 
\usepackage{appendix}

\newtheorem{theorem}{Theorem}[section]
\newtheorem{lemma}[theorem]{Lemma}

\theoremstyle{definition}
\newtheorem{definition}[theorem]{Definition}
\newtheorem{proposition}[theorem]{Proposition}

\theoremstyle{remark}
\newtheorem{remark}[theorem]{Remark}
\theoremstyle{notation}

\numberwithin{equation}{section}
\theoremstyle{corollary}
\newtheorem{corollary}[theorem]{Corollary}

\usepackage[bookmarks=false]{hyperref}

\newcommand{\Hom}{\mathrm{Hom}}

\newcommand{\Ho}{\mathrm{Ho}}

\begin{document}

\title[]{Derived noncommutative Zariski immersion and an equivalent reformulation of Friedlander-Milnor conjecture}

\author[]{Ilias Amrani}
\address{ Academic University of St-Petersburg, Russian Federation.}

\email{ilias.amranifedotov@gmail.com}





\begin{abstract}
We introduce the notion of derived formal noncommutative Zariski immersion for differential graded algebra with examples from topology. We illustrate the importance of such notion by reformulating the Friedlander-Milnor conjecture in terms of formal noncommutative Zariski immersions. This paper is based on the language developed by Dwyer, Greenless and Iyendar. 
 \end{abstract}
\maketitle
\section*{Notation}
\begin{itemize}
\item $p$ is a prime number and $\mathbf{F}_{p}=k$ is the field with $p$ elements.
\item If $G$ is a discrete group we denote the associated group algebra by $k[G]$.
\item $C_{\ast}(X;k)$ is the the singular chain complex associated to the space $X$ with coefficients in $k$.
\item $C^{\ast}(X;k)$ is the the singular cochain complex associated to the space $X$ with coefficients in $k$.
\item $\Ho(\mathsf{M})$ is the homotopy category of the model category $\mathsf{M}$.
\item If $A$ is a differential graded algebra (or ring spectra) we denote the stable model category of unbounded differential graded left $A$-modules by $A-\mathsf{Mod}$  cf \cite{EKMM}.
\item the derived tensor product is denoted by $-\otimes^{\mathbf{L}}-$ and the derived enriched hom functor (which is an unbounded chain complex) is denoted by $\mathbf{R}\Hom(-,-)$ cf. \cite{EKMM}.
\item $\mathsf{Top}$ is the symmetric monoidal closed model category of compactly generated weakly hausdorff spaces. 
\item $\Omega X$ is the space of loops of the pointed space $X$.
\item $BG$ is the classifying space of the the topological group $G$ and $EG$ is the universal contractible free $G$-space. 
\item $L_{p}\mathsf{Top}$ is the Bousfield localization of $\mathsf{Top}$ with respect to the homology functor $H_{\ast}(-;k)$. The functorial fibrant replacement is denoted by $X\rightarrow L_{p}X$. 
\item $\mathcal{E}_{\infty}$ is an operad whose algebras are $\mathcal{E}_{\infty}$-algebras and $\mathcal{A}_{\infty}$ is an operad whose algebras are $\mathcal{A}_{\infty}$-algebras.
\end{itemize}
\section{Noncommutative Zariski immersion}
 
 The notion of \textit{derived formal Zariski immersion} shows up in the context of $\mathcal{E}_{\infty}$-ring spectrum \cite{toen2003brave} seeking for a formulation of the Zariski topology in the setting of derived algebraic geometry. Our goal is to extend the definition in the noncommutative world and show that many interesting constructions in topology can be formulated using this notion. 
\begin{definition}\label{zariski}
Let $i: A\rightarrow B$ be a map of $\mathcal{A}_{\infty}$-differential graded algebras. The map $i$ is a derived formal Zariski immersion if the induced map of chain complexes  $B\otimes^{\mathbf{L}}_{A}B\rightarrow B$ is a quasi-isomorphism. 
  
\end{definition} 
It is not hard to see that our definition \ref{zariski} is equivalent to the one given in \cite[Definition 2.1.1]{toen2003brave} in the commutative setting. The definition is inspired form algebraic geometry cf \cite[Lemma 2.1.4]{toen2003brave}. If $spec(T)\rightarrow spec(R)$ is a Zariski open immersion between affine schemes then $R\rightarrow T$ is a derived formal Zariski immersion. The derived formal Zariski immersion are closely related to the smashing localization. For more details we refer to \cite{toen2003brave}.
 \begin{lemma}
 Let $i: A\rightarrow B$ be a a derived formal Zariski immersion then derived right adjoint functor between the triangulated categories
 $$\mathbf{R}i_{\ast}:\Ho(B-\mathsf{Mod})\rightarrow \Ho(A-\mathsf{Mod}) $$ 
 is fully faithful.
 \end{lemma}
 \begin{proof}
 Since $B\otimes^{\mathbf{L}}_{A}B\rightarrow B$ is a quasi-isomorphism and $B$ is a compact generator for $\Ho(B-\mathsf{Mod})$ we have that $B\otimes_{A}^{\mathbf{L}}\mathbf{R}i_{\ast}X\cong X$ for any object $X\in \Ho(B-\mathsf{Mod})$. Moreover, $\Ho(B-\mathsf{Mod})$ is a smashing localization of $\Ho(A-\mathsf{Mod})$ with respect to the left derived adjoint $B\otimes_{A}^{\mathbf{L}}-$. 
 \end{proof}
 \begin{remark}
 In the case when $\mathbf{R}i_{\ast}:\Ho(B-\mathsf{Mod})\rightarrow \Ho(A-\mathsf{Mod}) $ is fully faithful, the counit map $B\otimes_{A}^{\mathbf{L}}B\cong B\otimes_{A}^{\mathbf{L}}\mathbf{R}i_{\ast}B\rightarrow B$ is a quasi-isomorphism. Hence our definition \ref{zariski} is equivalent to the definition \cite[Definition 2.1.1]{toen2003brave} in the commutative case.   
 \end{remark}
 \section{Friedlander-Milnor Conjecture}
Let $G$ be a connected Lie group and let $G^{\delta}$ be the underlying discrete group, Milnor has conjectured \cite{milnor1983homology} that the natural group homomorphism 
 $$i: G^{\delta}\rightarrow G$$
 induces an isomorphism in cohomology of classifying space with coefficients in the finite field $\mathbf{F}_{p}$
$$i^{\ast}:H^{\ast}(BG;\mathbf{F}_{p})\rightarrow H^{\ast}(BG^{\delta};\mathbf{F}_{p}).$$ 
The conjecture is known to be true for solvable Lie groups in particular for $\mathbf{R}$ and the $n$-torus $T= S^1\times\dots\times S^1$ \cite{milnor1983homology}. Notice that the map $i^{\ast}$ is always injective. For more details we refer the reader to \cite[Chapter 5]{knudson2001homology}.   
  \begin{theorem}\label{milnor}
  The Friedlander-Milnor conjecture (for connected Lie group) is true if and only if the natural map of differential graded algebras $\mathbf{F}_{p}[G^{\delta}]\rightarrow C_{\ast}(G;\mathbf{F}_{p})$ is a derived formal Zariski immersion in the sense of \ref{zariski}. 
  \end{theorem}
  \begin{proof}
  Suppose that $i:\mathbf{F}_{p}[G^{\delta}]\rightarrow C_{\ast}(G;\mathbf{F}_{p})$ is a derived formal Zariski immersion, it induces a fully faithful functor 
  $$\mathbf{R}i_{\ast}:\Ho(C_{\ast}(G;\mathbf{F}_{p})-\mathsf{Mod})\rightarrow \Ho(\mathbf{F}_{p}[G^{\delta}]-\mathsf{Mod}) $$ 
  in particular 
  $$\mathbf{R}\Hom_{C_{\ast}(G;\mathbf{F}_{p})}(\mathbf{F}_{p},\mathbf{F}_{p})\simeq \mathbf{R}\Hom_{\mathbf{F}_{p}[G^{\delta}]}(\mathbf{F}_{p},\mathbf{F}_{p}).$$
  On another hand by \cite[Section 4.22]{dwyer2006duality} (or \cite{rothenberg19725} where the $E_{2}$-term of a converging spectral sequence is constructed), we have that $$H_{\ast}\mathbf{R}\Hom_{C_{\ast}(G;\mathbf{F}_{p})}(\mathbf{F}_{p},\mathbf{F}_{p})\simeq H^{\ast}(BG;\mathbf{F}_{p}),$$
  and 
  $$H_{\ast}\mathbf{R}\Hom_{\mathbf{F}_{p}[G^{\delta}]}(\mathbf{F}_{p},\mathbf{F}_{p})\simeq H^{\ast}(BG^{\delta};\mathbf{F}_{p}).$$
  Now, we suppose that the Friedlander-Milnor conjecture is true and we consider the following fibre sequence (of $G$-principle bundle) of topological spaces
  $$G\rightarrow G\times_{G^{\delta}}EG^{\delta}\rightarrow BG^{\delta}.$$ 
  In particular it induces the following fibre sequence 
  $$ G\times_{G^{\delta}}EG^{\delta}\rightarrow BG^{\delta}\rightarrow BG.$$
  Since we have supposed the Friedlander-Milnor is true, it implies by \cite[Lemma 5.1.2]{knudson2001homology} that $H_{\ast }(G\times_{G^{\delta}}EG^{\delta},\mathbf{F}_{p})$ has the homology of a point. Therefore by \cite[Theorem 3.1, Corollaire 2]{moore1959algebre}, we have a quasi-isomorphism between chain complexes
  $$ C_{\ast}(G\times_{G^{\delta}}EG^{\delta};\mathbf{F}_{p})\simeq C_{\ast}(G;\mathbf{F}_{p})\otimes^{\mathbf{L}}_{\mathbf{F}_{p}[G^{\delta}]}\mathbf{F}_{p}\simeq \mathbf{F}_{p}.$$
By hypothesis, $H_{\ast}(G,\mathbf{F}_{p})$ has finite total dimension and $H^{i} (BG;\mathbf{F}_{p})$ has a finite dimension for each $i\geq 0$ ($G$ is a connected Lie group). It follows that the augmentation map $C_{\ast}(G,\mathbf{F}_{p})\rightarrow \mathbf{F}_{p}$ is cosmall, in the sense that the thick subcategory of $\Ho(C_{\ast}(G;\mathbf{F}_{p})-\mathsf{Mod})$ generated by $\mathbf{F}_{p}$ contains $C_{\ast}(G;\mathbf{F}_{p})$ \cite[Section 5.6]{dwyer2006duality}. Therefore the equivalence 
  $C_{\ast}(G;\mathbf{F}_{p})\otimes_{\mathbf{F}_{p}[G^{\delta}]}^{\mathbf{L}}\mathbf{F}_{p}\simeq \mathbf{F}_{p}$ implies that 
  $$ C_{\ast}(G;\mathbf{F}_{p})\otimes_{\mathbf{F}_{p}[G^{\delta}]}^{\mathbf{L}}C_{\ast}(G;\mathbf{F}_{p})\simeq C_{\ast}(G;\mathbf{F}_{p}),$$
  since the derived functor $C_{\ast}(G;\mathbf{F}_{p})\otimes_{\mathbf{F}_{p}[G^{\delta}]}^{\mathbf{L}}-$ commutes with homotopy colimits and shifts.
  \end{proof}
  \begin{corollary}
  The natural map of $\mathcal{E}_{\infty}$-differential graded algebras $\mathbf{F}_{p}[T^{\delta}]\rightarrow C_{\ast}(T,\mathbf{F}_{p})$
  is a formal derived Zariski immersion for every prime $p$, where $T$ is a torus.  
   \end{corollary}
   \begin{proof}
   It follows directly from \ref{milnor} and the fact that Friedlander-Milnor conjecture is true for any torus \cite{milnor1983homology}.
   \end{proof}
   \begin{corollary}
   Let $G$ be a connected Lie group, the natural map of $\mathcal{A}_{\infty}$-differential graded algebras $\mathbf{F}_{p}[G^{\delta}]\rightarrow C_{\ast}(G,\mathbf{F}_{p})$ is a formal derived Zariski immersion for every prime $p$.
   \end{corollary}
   \begin{proof}
   It is a direct consequence of \ref{milnor} and \cite[Theorem 3.7]{amranimilnorzariski}.
   \end{proof}
  \section{$p$-Completion}
 Let $X$ a connected space and $X\rightarrow L_{p}X$ a fibrant replacement of $X$ in $L_{p}\mathsf{Top}$. For any connected sapce there exists a discrete group $G$ such that $BG\rightarrow X$ is an equivalence in $L_{p}\mathsf{Top}$ \cite{kan1976every}. 
  \begin{proposition}
  Let $X$ be a simply connected and $G$ as before, such that for any natural number $i$ the dimension of $H^{i}(X;\mathbf{F}_{p})$ is finite and the total dimension of $H_{\ast}(\Omega X;\mathbf{F}_{p})$ is finite. Then the map $G\simeq\Omega BG\rightarrow \Omega X$ induces the map  
  $\mathbf{F}_{p}[G]\rightarrow C_{\ast}(\Omega X;\mathbf{F}_{p})$ of $\mathcal{A}_{\infty}$-differential graded algebras which is a derived formal Zariski immersion. 
  \end{proposition}
 \begin{proof}
 First, we recall that any topological monoid group-like is equivalent to a honest topological group. By the same arguments as in \ref{milnor}, we have that $H_{\ast}(\Omega X\times_{G}EG;\mathbf{F}_{p})=0$ for $\ast > 0$ i.e. has homology of a point. Using the finiteness conditions and the same arguments as in \ref{milnor}, we conclude that $$C_{\ast}(\Omega X;\mathbf{F}_{p})\otimes_{\mathbf{F}_{p}[G]}^{\mathbf{L}}C_{\ast}(\Omega X;\mathbf{F}_{p})\simeq C_{\ast}(\Omega X;\mathbf{F}_{p}).$$
 \end{proof}
 \bibliographystyle{plain}
\bibliography{solid}
\end{document}